\documentclass[preprint,12pt]{elsarticle}

\usepackage{latexsym}
\usepackage{amssymb}
\usepackage{amsmath,amsthm,amsfonts,mathrsfs,amsbsy}
\usepackage[mathscr]{eucal}
\usepackage{times}
\usepackage{tikz}
\usepackage{qtree}
\usepackage{hyperref}
\hypersetup{
     colorlinks=true,
     linkcolor=blue,
     filecolor=blue,
     citecolor = black,      
     urlcolor=cyan,
     }
\usepackage{cleveref}
\usepackage[all,cmtip]{xy}

\usepackage{enumitem}

\usepackage{tikz}
\usepackage{graphicx}
\usepackage{tikz-cd}
\usetikzlibrary{matrix,arrows,decorations.pathmorphing}
\usepackage[small,nohug,heads=vee]{diagrams}
\diagramstyle[labelstyle=\scriptstyle]
\usepackage{xcolor}

\theoremstyle{plain}
\newtheorem{thm}{Theorem}[section]
\newtheorem{lem}[thm]{Lemma}
\newtheorem{prop}[thm]{Proposition}
\newtheorem{cor}[thm]{Corollary}
\theoremstyle{definition}

\newtheorem*{theorem*}{Theorem}

\newcommand{\Q}{\mathbb{Q}}
\newcommand{\N}{\mathbb{N}}
\newcommand{\Z}{\mathbb{Z}}
\DeclareMathOperator{\pt}{pt}

\usepackage{relsize}

\newtheorem{dfn}[thm]{Definition}
\newtheorem{obs}[thm]{Remark}

\begin{document}
\begin{frontmatter}



\title{On the Cantor and Hilbert Cube Frames and the Alexandroff-Hausdorff Theorem}


\author{F. Avila, J. Urenda, A. Zald\'ivar}

\address{}

\begin{abstract}


The aim of this work is to give a pointfree description of the Cantor set.  It can be shown (see, e.g. \cite{Alain}) that the Cantor set is homeomorphic to the $p$-adic integers $\Z_p:=\{x\in\Q_p: |x|_p\leq 1\}$ for every prime number $p$. To give a pointfree description of the Cantor set, we specify the frame of $\Z_p$ by generators and relations. We use the fact that the open balls centered at integers
generate the open subsets of $\Z_p$ and thus we think of them as the basic generators; on this poset we impose some relations and then the resulting quotient is the frame of the Cantor set $\mathcal{L}(\mathbb{Z}_{p})$.
 
 A topological characterization of it is given by Brouwer's Theorem \cite{Brouwer}: The Cantor set is the unique totally disconnected, compact metric space with no isolated points. 
We prove that $\mathcal{L}(\Z_p)$ is a spatial frame whose space of points is homeomorphic to $\Z_p$. In particular, we show with pointfree arguments that $\mathcal{L}(\Z_p)$ is $0$-dimensional, (completely) regular, compact, and metrizable (it admits a countably generated uniformity). Moreover, we show that the frame $\mathcal{L}(\mathbb{Z}_{p})$ satisfies $\displaystyle{cbd_{\mathcal{L}(\mathbb{Z}_{p})}(0)=0}$, where $\displaystyle{cbd_{\mathcal{L}(\mathbb{Z}_{p})}:\mathcal{L}(\mathbb{Z}_{p}) \to \mathcal{L}(\mathbb{Z}_{p})}$, defined by $$\displaystyle{cbd_{\mathcal{L}(\mathbb{Z}_{p})}(a)=\bigwedge\big\{x\in \mathcal{L}(\mathbb{Z}_{p}) \mid x\leq a \text{ and } (x\to a)=a\big\}}$$
is the Cantor-Bendixson derivative (see, e.g. \cite{MR3363833}). It follows that a frame $L$ is isomorphic to $\mathcal{L}(\mathbb{Z}_{p})$ if and only if $L$ is a $0$-dimensional compact regular metrizable frame with $cbd_L(0)=0$. Finally, we give a point-free counterpart of the Hausdorff-Alexandroff Theorem which states  that \emph{every compact metric space is a continuous image of the Cantor space} (see, e.g. \cite{Alexandroff} and \cite{Hausdorff}). We prove the point-free analogue: if $L$ is a compact metrizable frame, then there is an injective frame homomorphism from $L$ into $\mathcal{L}(\mathbb{Z}_{2})$.
\end{abstract}





\end{frontmatter}


\section{Introduction}\label{intro}
Frames (locales, complete Heyting algebras) appeared as an algebraic manifestation of topological spaces. Indeed, for any topological space $S$ its topology $\Omega(S)$ is a frame. The category \textbf{Frm} consists of frames as objects and $\bigvee,\wedge$-preserving functions as morphisms, we have a contravariant adjunction (see \cite{Johnstone82} or \cite{PicadoPultr} for details):
\[\begin{tikzcd}
\textbf{Top} \arrow[r, "\Omega(\_)"{name=F}, bend left=25] &  \textbf{Frm} \arrow[l, "\pt(\_)"{name=G}, bend left=25] \arrow[phantom, from=F, to=G, "\dashv" rotate=-90] 
\end{tikzcd}\]
Several \emph{point-sensitive} constructions  have a frame theoretical counterpart. For example, Joyal in \cite{Joyal} produces a point-free construction of the reals where a detailed construction can be found in \cite[Chapter IV]{Johnstone82}. Banaschewski in \cite{Banaschewski} gives another construction of the frame of the reals using generators and relations where aspects of the reals and its rings of continuous functions can be translated into this setting ( see \cite[Chapter XIV]{PicadoPultr}). Later, in \citep{avila2020frame} the first author introduces the frame of $p$-adic numbers defined with relations on the partial ordered set of open balls of the corresponding ultrametric of $\mathbb{Q}_{p}$. As in the case of the reals, the first author explores the rings of $p$-adic continuous functions.
This manuscript explores a modification of the construction of the frame of $p$-adic numbers $\mathcal{L}(\mathbb{Q}_{p})$ leading to the the frame of the Cantor set. 
Set $\displaystyle{|\Z|:=\{p^{-n+1}\mid n\in \N\}}$ and let $\mathcal{L}(\Z_p)$ be the frame generated by the elements $B_r(a)$, where $a\in \Z$ and $r\in|\Z|$, subject to the following relations:
\begin{itemize}
\item[(Q$_1$)] $B_r(a)\wedge B_s(b)=0 $ whenever $|a-b|_p\geq r$ and $s\leq r$.
\item[(Q$_2$)] $1=\bigvee \big\{B_r(a): a\in \Z,\ r\in|\Z| \big\}$.
\item[(Q$_3$)] $B_r(a)=\bigvee \big\{B_s(b) : |a-b|_p<r, \, s<r,\, b\in\Z,\, s\in|\Z|\big\}$.
\end{itemize}

Among several properties of this frame, we prove that it is spatial and isomorphic to the topology given by the ultrametric of the $p$-adic integers $\mathbb{Z}_{p}$. Our main motivation for the study of this frame comes from a remarkable topological fact:

\begin{theorem*}(Hausdorff-Alexandroff)

Every compact metric space is a continuous image of the Cantor space
\end{theorem*}

In this manuscript we provide a point-free version of this theorem (Theorem \ref{Alex-Haus}).
There exists many proofs of the Hausdorff-Alexandroff theorem in the literature, we take inspiration of a proof located in \cite{guide2006infinite}; we implement a \emph{Hilbert-cube trick} characterizing every compact metric frame as a closed quotient of the Hilbert cube frame and showing an injective frame morphism from the Hilbert cube frame to the frame of $2$-adic integers.

In section \ref{prel} we present the necessary background used freely in our investigation.

Section \ref{Cantor} contains the presentation of the Cantor frame $\mathcal{L}(\mathbb{Z}_{p})$. We give fundamental properties of this frame such as spatiality and perfectness. Also we observe that $\mathcal{L}(\mathbb{Z}_{p})$ is a complete uniform metrizable frame. At the end of this section, we give a direct characterization of $\mathcal{L}(\mathbb{Z}_{p})$ (Theorem \ref{c2}).

In section \ref{haus} we give a point-free version of the Alexandroff-Hausdorff theorem. We provide an injective morphism from the frame of $[0,1]$ into the frame of $2$-adic integers via the dyadic rationals. We thus extent this morphism to an injective frame morphism from the Hilbert cube frame.

\section{Preliminaries}\label{prel}
A \emph{frame} is a complete lattice $L$ satisfying the distributive law
\[ \big(\bigvee A \big)\wedge b=\bigvee\{a\wedge b \mid a\in A\}\] for any subset $A\subseteq L$ and any $b\in L$. We denote its top element by $1$ and its bottom element by $0$.
A frame is called \emph{spatial} if it is isomorphic to $\Omega(X)$ for some space $X$. A \emph{frame homomorphism} is a map $h\colon L\rightarrow M$ satisfying 
\begin{equation*} \label{frmhom}
h(0)=0, \, h(1)=1, \, h(a\wedge b)=h(a)\wedge h(b),\text{ and } h\big( \bigvee_{i\in J} a_i\big)=\bigvee_{i\in J} h(a_i), \text{for any set} J
\end{equation*}
where $L$ and $M$ are frames. Also, $h$ is \emph{dense} if $h(x)=0$ implies $x=0$.

The category \textbf{Frm} of frames, with objects frames and morphisms frame moprhisms. We have the natural functor $\Omega\colon \textbf{Top}\rightarrow\textbf{Frm}$ from the category \textbf{Top} of topological spaces to the category \textbf{Frm} of frames is contravariant; we define the category of locales \textbf{Loc} as the dual category of \textbf{Frm}. Since a frame homomorphism $h\colon M\to L$ preserves all suprema, it has a uniquely associated right Galois adjuntion $h_*\colon L\to M$. Morphisms in \textbf{Loc} are represented as the infima preserving $f\colon L\to M$ such that the corresponding left adjoint $f^*\colon M\to L$ preserve finite meets. Such maps correspond to the morphims in \textbf{Loc}.

Let $L$ be a frame and set $\pt(L):=\{\alpha\colon L\to \boldsymbol{2}\mid \alpha \textit{ is a frame homomorphism}\}$ and for $a\in L$, set $U_a=\{\alpha\in\pt L \mid \alpha(a)=1\}$. It is straightforward to verify that 
\[U_0=\varnothing, \,\,\, U_1=\pt L, \,\,\, U_{a\wedge b}=U_a \cap U_b, \,\,\, U_{\bigvee a_i}=\bigcup U_{a_i}\]
and hence $\{U_a\mid a\in L \}$ is indeed a topology on $\pt(L)$. Moreover, for a frame homomorphism $h\colon L\to M$, define $\pt(h)\colon\pt(M) \to \pt(L)$ by setting $\alpha \mapsto \alpha h$. Since we have $(\pt (h))^{-1}[U_a]=U_{h(a)}$, $\pt(h)$ is a continuous map. This shows that $\pt$ is contravariant functor from $\textbf{Frm}$ to $\textbf{Top}$.

The \emph{negation} or \emph{pseudocomplement} of an element $a$ in a frame $L$ is the element $\neg a=\bigvee\{x\in L \mid x\wedge a=0\}$. An element $a$ in $L$ is said to be \emph{complemented} or \emph{clopen} if there is an element $b\in L$ such that $b\wedge a=0$ and $b\vee a=1$; this element $b$ is called the complement of $a$ in fact $b=\neg a$. A frame $L$ is said to be \emph{$0$-dimensional} if $a=\bigvee\big\{b\mid b\leq a, \, b \text{ is complemented}\big\}$ for each $a\in L$. A $0$-dimensional frame $L$ is called \emph{ultranormal} if whenever $a\vee b=1$, there is a complemented element $c\in L$ such that $c\leq a$ and $c'\leq b$. A \emph{partition} of a frame $L$ is a subset $P\subseteq L\smallsetminus\{0\}$ satisfying $\bigvee P=1$ and $a\wedge b=0$ whenever $a,b\in P, \, a\neq b$. A $0$-dimensional frame $L$ is said to be \emph{ultraparacompact} if whenever $C$ is a cover of $L$, there is a partition $P$ that refines $C$.

We provide a brief description of the coproducts of frames as in \cite{PicadoPultr}. Let $\{F_{i}\mid i\in\EuScript{K}\}$ be a family of frames and consider the coproduct of semilattices $\prod'_{i\in\EuScript{K}}F_{i}$ with $\kappa_{j}\colon F_{j}\rightarrow\prod'_{i\in\EuScript{K}}F_{i}$ for each $j\in\EuScript{K}$, see \cite[IV.4]{PicadoPultr}. Then, we apply the down-set functor $\mathcal{D}(\prod'_{i\in\EuScript{K}}F_{i})$ and observe that we have a morphism
$\alpha\colon \prod'_{i\in\EuScript{K}}F_{i}\rightarrow\mathcal{D}(\prod'_{i\in\EuScript{K}}F_{i})$, $\alpha(x)=\downarrow x$. Take the relation $R$ on $\mathcal{D}(\prod'_{i\in\EuScript{K}}F_{i})$ as in \cite[IV. 4.3]{PicadoPultr} to obtain the frame $\mathcal{D}(\prod'_{i\in\EuScript{K}}F_{i})/R=\bigoplus_{i}F_{i}$. Note that we have a quotient morphism $\mu\colon\mathcal{D}(\prod'_{i\in\EuScript{K}}F_{i})\rightarrow\bigoplus_{i}F_{i}$ given by the nucleus induced by the saturated relation $R$. For each $j\in\EuScript{K}$ we have a frame morphism \[\iota_{j}=\mu\alpha\kappa_{j}\colon F_{j}\rightarrow\bigoplus_{i}F_{i},\] 
the couple $(\bigoplus_{i}F_{i},\iota_{i})_{i}$ constitute the coproduct of the family. Whenever $F=F_{i}$ for all $i\in\EuScript{K}$ we write $^{ \EuScript{K}}F$.

The frame $\mathcal{L}[0,1]$ is an upper section of the frame $\mathcal{L}(\mathbb{R})$, the frame of the reals. This frame is generated by the partial ordered sets of open intervals $(p,q)$ with $p,q\in\mathbb{Q}$ (as in the classical case $\mathcal{L}(\mathbb{R})$ is a completion of that partial ordered set) module some relations. The frame of $[p,q]$ is defined as $\uparrow\hspace{-.05in}((-,p)\vee (q,-))$, for details see \cite[Chapter XIV]{PicadoPultr}. Another useful way to construct $\mathcal{L}(\mathbb{R})$, and hence $\mathcal{L}[0,1]$, is given in \cite[XIV 1.2.2]{PicadoPultr}, we use that equivalent construction on section \ref{haus} for the sake of completeness we present it here.

\begin{thm}\label{reals}
The frame of the reals $\mathcal{L}(\mathbb{R})$ is generated by $(p,-)$ and $(-,q)$ with $p,q\in\mathbb{Q}$ subject to the following relations.

\begin{itemize}
\item[(1)] $(p,-)\wedge(-,q)=0$ whenever $p\geq q$.
\item[(2)] $(p,-)\vee(-,q)=1$ whenever $p<q$.
\item[(3)] $(p,-)=\bigvee\{(r,-)\mid r>p\}$.

\item[(4)] $(-,q)=\bigvee\{(-,s)\mid s<q\}$.
\item[(5)] $\bigvee\{(p,-)\mid p\in\mathbb{Q}\}=1$

\item[(6)] $\bigvee\{(-,q)\mid q\in\mathbb{Q}\}=1$
\end{itemize}
\end{thm}

Note that replacing the set of rational numbers $\mathbb{Q}$ by a countable dense subset $D$ of the real numbers in the above definition results in a frame, $\mathcal{L}(\mathbb{R})_D$, isomorphic to $\mathcal{L}(\mathbb{R})$. 

Given two real numbers $p,q \in D$ with $p < q$ we consider the frame of the closed interval $\mathcal{L}[p,q]$ as the closed quotient $\uparrow(p,q)$.

Besides, we recall the dyadic rationals consisting of all rational numbers $p/q$ with $q = 2^n$ for some nonnegative integer $n$. Clearly, this is a dense subset of the real numbers (see \cite{mcdonald1999course} ). 

In section 4 we make use of the dyadic rationals when constructing a frame morphism between the $\mathcal{L}[0,1]$ and $\mathcal{L}(\mathbb{Z}_2)$ where we define $\mathcal{L}[0,1]$ with the dyadic rationals instead of the rational numbers.

In Cantor's original diagonalization argument on the uncoutability of the real line, he warned about real numbers with more than one decimal expansion, we detail the binary case in the following.

\begin{lem}
Let $(u_i)$ be a sequence with values in $[0,1]$ then $\sum_{i > 0} \frac{u_i}{2^i} = 1$ if and only if $u_i = 1$ for all $i$.
\end{lem}
\begin{proof}
The converse follows from the convergence of the geometric sequence.
Now suppose that $\sum_{i > 0} \frac{u_i}{2^i} = 1$ then $\sum_{i > 0}\frac{1}{2^i} - \sum_{i > 0}\frac{1}{2^i} = \sum_{i > 0}\frac{1-u_i}{2^i} = 0$, since $u_i \in [0,1]$, then $1-u_i \in [0,1]$, thus $1-u_i = 0$ for all $i$; equivalently, $u_i=1$.
\end{proof}

\begin{lem}
Let $x \in [0,1]$ then and let $(u_i)$ and $(u'_i)$ be two binary sequences where
$$x = \sum_{i > 0}\frac{u_i}{2^i}=\sum_{i > 0}\frac{u'_i}{2^i}$$
then either $(u_i)=(u'_i)$ or there exists an nonnegative integer $g$ where $u_g=1$, $u'_g=0$, and for all $i>g$, $u_i=0$ and $u'_i=0$ or the same with $u$ and $u'$ exchanged.
\end{lem}
\begin{proof}
If $(u_i)=(u'_i)$ we are done. So assume that $(u_i)\neq(u'_i)$, by the well-ordering principle, there is a minimum non-negative integer $g$ such that $u_g\neq u'_g$, for the sake of argument, let us say that $u_g=1$ and $u'_g=0$ then
$$0 = \sum_{i > 0}\frac{u_i}{2^i}-\sum_{i > 0}\frac{u'_i}{2^i}=\frac{1}{2^g} + \sum_{i>g}\frac{u_i-u'_i}{2^i}$$
factoring $\frac{1}{2^g}$ from the last two terms we get that $$1-\sum_{i>0}\frac{u'_{i+g}-u_{i+g}}{2^i}=0$$ by the above lemma it follows that for all
$i > g$, $u'_i-u_i = 1$, since both are binary sequences it follows that $u'_i = 1$ and $u_i = 0$.
\end{proof}

For last we mention that the elements in $\mathcal{L}[p,q]$ are denoted by $(r,s)$ to mean $((-,p)\vee(q,-))\vee(r,s)$.


\section{The frame of the Cantor set $\mathcal{L}(\Z_p)$ and their Properties} \label{Cantor}

Recall that $\mathbb{Q}_p$ is the completion of $\mathbb{Q}$ with respect to its ultra-metric. In \cite{avila2020frame}, the frame of the $p$-adic numbers $\mathcal{L}(\mathbb{Q}_p)$ is defined by generators and relations among open balls centered on rationals numbers. Besides, $\mathbb{Z}_{p}$ is the ring of integers of $\mathbb{Q}_{p}$ and it is a topological subspace which is homeomorphic to the Cantor set for every $p$, see \citep{Alain}. Following the same construction in \cite{avila2020frame} we introduce the free frame of the $p$-adic integers by the replacing rationals with integers.

\begin{dfn}\label{pad} Let $\mathcal{L}(\mathbb{Z}_p)$ be the frame generated by the elements $B_r(a)$, where $a\in \mathbb{Z}$ and $\displaystyle{r\in |\mathbb{Z}|:=\{p^{-n+1}\mid n\in \mathbb{N}\}}$, subject to the following relations:
\begin{itemize}
\item[(Q$_1$)] $B_r(a)\wedge B_s(b)=0 $ whenever $|a-b|_p\geq r$ and $s\leq r$.
\item[(Q$_2$)] $1=\bigvee \big\{B_r(a): a\in \mathbb{Z},\ r\in|\mathbb{Z}| \big\}$.
\item[(Q$_3$)] $B_r(a)=\bigvee \big\{B_s(b) : |a-b|_p<r, \, s<r,\, b\in\mathbb{Z},\, s\in|\mathbb{Z}|\big\}$.
\end{itemize} 
\end{dfn}
\textbf{Notation:} We will use $B_r\langle a\rangle$ to denote the open ball in $(\mathbb{Z},|\cdot|_p)$ centered at $a$ with radius $r\leq 1$, i.e., $B_r\langle a\rangle:=\{x\in\mathbb{Z} : |x-a|_p<r\}$, and we will use $S_r\langle a \rangle$ to denote the open ball in $(\mathbb{Z}_p,|\cdot|_p)$ centered at $a\in\mathbb{Z}$ with radius $r\leq1$; that is, $S_r\langle a\rangle:=\{x\in\mathbb{Z}_p : |x-a|_p<r\}$.\\  

Many frame-theoretic properties of $\mathcal{L}(\mathbb{Q}_{p})$ are inherited to $\mathcal{L}(\mathbb{Z}_{p})$. Moreover, most of these proofs are similar

\begin{obs} It can be shown that (see \cite{avila2020frame}), for $a,b\in \mathbb{Z}$ and $r,s\in |\mathbb{Z}|$ with $s\leq r$, we have:
\begin{itemize}
\item[(R$_1$)] $B_s(b)\leq B_r(a)$ whenever $|a-b|_p < r$ and $s\leq r$. In particular, we have $B_r(a)=B_r(b)$ whenever $|a-b|_p<r$.
\item[(R$_2$)] $B_r(a)\wedge B_s(b)=B_{s}(a)=B_{s}(b)$ whenever $|a-b|_p<s$. 
\item[(R$_3$)] $B_r(a)\vee B_s(b)=B_{r}(b)=B_{r}(a)$ whenever $|a-b|_p<r$. 
\item[(R$_4$)] If $|a-b|_p<r$, then either $B_s(b)\leq B_r(a)$ or $B_s(b)\geq B_r(a)$.
\item[(R$_5$)] $B_r(a)=\bigvee \{B_{r/p}(a+xp^{n+1})\mid x=0,1,\dots,p-1\}$.
\end{itemize}
\end{obs}

In this part, we show that the frame $\mathcal{L}(\Z_p)$ is $0$-dimensional, (completely) regular, and compact. In particular, we will see that $\mathcal{L}(\Z_p)$ is ultraparacompact and ultranormal (see, e.g. \cite{Prolla1}). We begin noting an important fact: each generator is complemented.
\begin{prop}
Let $B_r(a)\in\mathcal{L}(\Z_p)$ be a generator. Then $B_r(a)$ is complemented (clopen) and $B_r(a)'=\bigvee\{B_r(b)\,|\, |a-b|_p\geq r\}$.
\end{prop}
\begin{proof} By (Q$_1$), we have
\[B_r(a)\wedge \bigvee\{B_r(b)\mid |a-b|_p\geq r\}=\bigvee\{B_r(a)\wedge B_r(b)\mid |a-b|_p\geq r\}=0.\] 
Now, let $c\in \Z$ and $t\in|\Z|$. First, assume $t\leq r$. Then $|a-c|_p\geq r$ implies that $B_t(c)\leq \bigvee\{B_r(b)\,|\, |a-b|_p\geq r\}$, and $|a-c|_p< r$ implies $B_t(c)\leq B_r(a)$. Thus, for any $c\in\Z$ and $t\in|\Z|$ with $t\leq r$, we have \[B_t(c)\leq B_r(a)\vee \bigvee\{B_r(b)\,|\, |a-b|_p\geq r\}.\] Note that if $t>r$, then (R$_5$) implies that $\displaystyle{B_t(c)=\bigvee_{i=1}^n \big\{B_r(c_i) \big\}}$ for some $c_i\in\Z$, and by the above argument, we have $\displaystyle{B_t(c)\leq B_r(a)\vee \bigvee\{B_r(b)\,|\, |a-b|_p\geq r\}}$. It follows that, in either case, $\displaystyle{B_t(c)\leq B_r(a)\vee \bigvee\{B_r(b)\,|\, |a-b|_p\geq r\}}$, and by (Q$_2$) we have that $\displaystyle{B_r(a)\vee \bigvee\{B_r(b)\mid |a-b|_p\geq r\}=1.}$
\end{proof}
\begin{cor}
The frame $\mathcal{L}(\Z_p)$ is $0$-dimensional.
\end{cor}
\begin{cor}
$\mathcal{L}(\Z_p)$ is completely regular.
\end{cor}
\begin{proof}
First, note that if $B_s(b)\leq B_r(a)$, then $B_s(b)'\geq B_r(a)'$ and
\begin{align*} 
B_r(a)\vee B_s(b)'\geq B_r(a)\vee B_r(a)'=1. 
\end{align*}
Therefore $B_s(b)\prec B_r(a)$. In particular, $B_r(a)\prec B_r(a)$.
It follows immediately that $\prec$ interpolates and since $\prec \hspace{-.05in} \prec$ is the largest interpolative relation contained in $\prec$, we must have that $\prec = \prec \hspace{-.05in} \prec$. Then, by remark (Q$_3$), we have that \begin{center} $B_r(a)= \bigvee\{B_s(b):B_s(b)\leq B_r(a)\}=\bigvee\{B_s(b):B_s(b)\prec \hspace{-.05in} \prec B_r(a)\}.$ \end{center}
\end{proof}

Let $L$ be a $0$-dimensional frame. Recall that $L$ is said to be \emph{ultranormal} if whenever $a\vee b=1$, there is a complemented element $c\in L$ such that $c\leq a$ and $c'\leq b$, and $L$ is \emph{ultraparacompact} if whenever $C$ is a cover of $L$ there is a partition $P$ that refines $C$. 

\begin{prop}
The frame $\mathcal{L}(\Z_p)$ is ultranormal and ultraparacompact.
\end{prop}
\begin{proof}
Since every generator in $\mathcal{L}(\Z_p)$ can be written as a finite join of disjoint generators, it follows immediately that $\mathcal{L}(\Z_p)$ is ultraparacompact.
Now, to verify that $\mathcal{L}(\Z_p)$ is ultranormal, let $a,b\in \mathcal{L}(\Z_p)$ and suppose that $a\vee b=1$. Set $C=\{x\in \mathcal{L}(\Z_p)\mid x\prec a\}$, then $\bigvee C\vee b=a\vee b=1$ by regularity of $\mathcal{L}(\Z_p)$. That is, $C\cup \{b\}$ is a cover of $\mathcal{L}(\Z_p)$. Since $\mathcal{L}(\Z_p)$ is ultraparacompact, this cover has a refinament which is a clopen partition of $\mathcal{L}(\Z_p)$, say $\{a_i\mid i\in I\}$. Setting $c=\bigvee\{a_i\mid a_i\leq b\}$, we obviously have that $c$ is complemented, $c\leq b$, and $c'\leq a$. Therefore, $\mathcal{L}(\Z_p)$ is ultranormal.
\end{proof}

Next we will see that $\mathcal{L}(\mathbb{Z}_{p})$ is a compact frame by identifying this frame with a compact frame obtained from $\mathcal{L}(\mathbb{Q}_{p})$. Also, we will consider a second constructive proof inspired by $n$-ary rooted trees.


\begin{prop}\label{clsubloc}
$\mathcal{L}(\mathbb{Z}_{p})$ is isomorphic to a closed sublocale of $\mathcal{L}(\mathbb{Q}_{p})$.
\end{prop}

\begin{proof}
Consider the map $\pi: \mathcal{L}(\mathbb{Z}_{p}) \to \,\uparrow \! B_p(0)'$ defined on generators  $$\pi(B_r(a))=B_r(a)\vee B_p(0)'$$ for each $a \in \mathbb{Z}$ and $r \in |\mathbb{Z}| = \{p^{-n+1} \mid n \in \mathbb{N}\}$. Clearly, $\pi$ is an onto frame morphism.
Now let $a_1,a_2 \in \mathbb{Z}$ and $r_1,r_2 \in \lvert \mathbb{Z} \rvert$ and suppose that $\pi(B_{r_1}(a_1))=\pi(B_{r_2}(a_2))$, then $B_{r_1}(a_1) \vee B_p(0)' = B_{r_2}(a_2)\vee B_p(0)$ so $$[B_{r_1}(a_1) \vee B_p(0)'] \land B_p(0) = [B_{r_2}(a_2)\vee B_p(0)] \land B_p(0)$$ distributing $$[B_{r_1}(a_1) \land B_p(0)] \vee [B_p(0) \land B_p(0)'] = [B_{r_2}(a_2) \land B_p(0)] \vee [B_p(0) \land B_p(0)']$$ reducing $$B_{r_1}(a_1) \land B_p(0)=B_{r_2}(a_2) \land B_p(0)$$ thus $B_{r_1}(a_1) = B_{r_2}(a_2)$ since $B_{r_1}(a_1), B_{r_2}(a_2) \le B_p(0)$, thus $r_1=r_2$ and $|a_1-a_2|_p < r_1=r_2$ so $B_{r_1}(a_1) = B_{r_2}(a_2)$ in $\mathcal{L}(\mathbb{Z}_{p})$. Therefore, $\pi$ is a frame isomorphism.
\end{proof}

\begin{cor}
$\uparrow \! B_p(0)'$ is compact.
\end{cor}

\begin{proof}
Let $a \in \mathbb{Z}$ and $r \in |\mathbb{Z}|$, then $$B_r(a) \prec B_p(0) \ll 1 \text{ in } \mathcal{L}(\mathbb{Q}_{p})$$ and let $\{x_i\}_{i \in J}$ be a cover in $\uparrow \! B_p(0)'$,  by continuity of $\mathcal{L}(\mathbb{Q}_{p})$, see \citep{Avila}, there are $i_1, \ldots, i_n$ in $J$ such that $B_p(0) \le x_{i_1}\lor \cdots \lor x_{i_n}$, thus $$1 = B_p(0) \lor B_p(0)' \le x_{i_1}\lor \cdots \lor x_{i_n} \lor B_p(0)'.$$ Therefore $\{x_{i_1},\ldots,x_{i_n}\}$ is a finite subcover.
\end{proof}

\begin{cor}\label{cpt}
$\mathcal{L}(\mathbb{Z}_{p})$ is compact.
\end{cor}

\begin{cor} $\mathcal{L}(\mathbb{Z}_{p})$ is spatial.
\end{cor}

A second proof of compactness based on a choice-free version of K\"{o}nig's lemma is presented. Recall that a {\it tree} is a partially ordered set of $(T,\le_T)$ with a unique least element called the {\it root} in which the predecessors of every element are well-ordered by $\le_T$. Moreover, a tree is $n$-ary if each element has at most $n$ immediate successors. For $x,y \in T$, $x$ is a {\it descendant} of $y$ iff $x <_T y$.

\begin{proof}
Let $\mathcal{C} \subseteq \mathcal{L}(\mathbb{Z}_p)$ where $\bigvee \mathcal{C} = 1$, if $1 \in \mathcal{C}$ there is nothing to prove; otherwise, consider the $p$-ary tree $(T,\le_T)$ where $x \in T$ if and only if $x = 1$ or $x = B_{p^{-n}(a)}$ for some $0 \le a < p^n$ and $B_{p^{-n}(a)} \not\le u$ for all $u \in \mathcal{C}$, and $\le_T$ is the reverse partial order inherited by the frame $\mathcal{L}(\mathbb{Z}_p)$. Now, for the sake of contradiction assume that $T$ is infinite. Note that $1 \in T$, call $b_1$ the immediate successor of $1$ with the smallest $0 \le i < p$ such that $B_{p^{-1}}(i)$ has infinitely many descendants in $T$. For $n \ge 1$, call $b_{n+1}$ the immediate successor of $b_n = B_{p^{-n}}(a)$ with the smallest $0 \le i < p$ where $B_{p^{-(n+1)}}(a + ip^n)$ has infinitely many descendants in $T$. This defines a sequence $(b_n)_{n \ge 1}$ in $\mathcal{L}(\mathbb{Z}_p)$. Now, set $h:\mathcal{L}(\mathbb{Z}_p) \to \mathbf{2}$ defined on generators by $h(B_r(a)) = 1$ if $B_r(a)=b_n$ for some $n \ge 1$; otherwise $h(B_r(a))=0$. Note that $h$ is an onto frame morphism, thus $h(\mathcal{C})$ is a cover of $\mathbf{2}$, but for each $u \in \mathcal{C}$, $h(u)=0$ since for each $n \ge 1$, $b_n \not\le u$, contradiction. Then $T$ is finite, so there is a smallest positive $n$ such that $B_{p^{-n}}(a) \not\in T$ for all $0 \le a < p^n$. Now for each $0 \le i < p^k$ pick $c_i$ in $\mathcal{C}$ with $B_{p^{-n}}(i) \le c_i$, since $1=\bigvee_{i=0}^{p^n-1}B_{p^{-n}}(i) \le \bigvee_{i=0}^{p^n-1} c_i$ we have that $\{c_i\}_{i=0}^{p^n-1} \subseteq \mathcal{C}$
\end{proof}

We now proceed to a categorical characterization of $\mathcal{L}(\mathbb{Z}_p)$ based on compactness, denseness, and metrizability. Note that, by (R$_5$), there are $x_1,\dots, x_p\in \Z$ such that $$\displaystyle{B_1(0)=B_{p^{-1}}(x_1)\vee \cdots \vee B_{p^{-1}}(x_p)}$$ with $B_{p^{-1}(x_i)}\wedge B_{p^{-1}}(x_j)=0$; and for each $x_i$, there are $x_{i1},\dots,x_{ip}\in \Z$ such that $\displaystyle{B_{p^{-1}}(x_i)=B_{p^{-2}}(x_{i1})\vee \cdots \vee B_{p^{-2}}(x_{ip})}$. Continuing this way, we can form the following tree\\

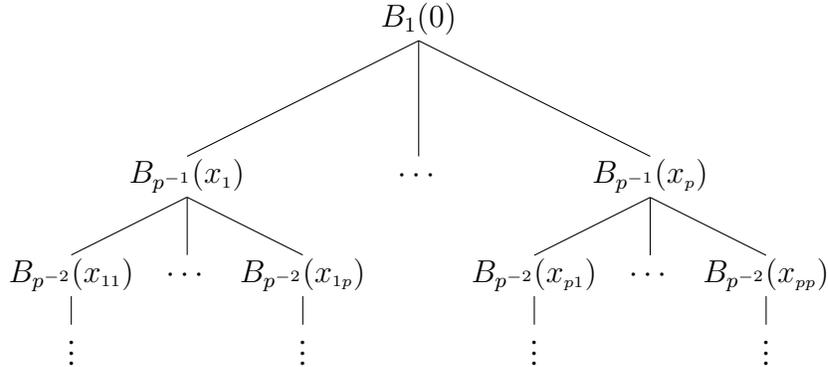
\Tree[.$B_1(0)$ [.$\displaystyle{B_{p^{-1}}(x_{\scriptscriptstyle 1})}$ [.$\displaystyle{B_{p^{-2}}(x_{\scriptscriptstyle 11})}$ $\vdots$ ] 
            $\cdots$   [.$\displaystyle{B_{p^{-2}}(x_{\scriptscriptstyle 1p})}$ $\vdots$ ] ] 
       $\cdots$ [.$\displaystyle{B_{p^{-1}}(x_{\scriptscriptstyle p})}$ [.$\displaystyle{B_{p^{-2}}(x_{\scriptscriptstyle p1})}$ $\vdots$ ] $\cdots$
                [.$\displaystyle{B_{p^{-2}}(x_{\scriptscriptstyle pp})}$ $\vdots$ ]]]\vspace{.2in}\\
and, by (R$_1$)$-$ (R$_5$), we see that each $B_r(a)\in \mathcal{L}(\mathbb{Z}_{p})$ is one of the elements of this tree. It is immediate that, for each $B_r(a)$, the interval $[0,B_r(a)]$ is isomorphic to $\mathcal{L}(\mathbb{Z}_{p})$.
Finally, recall that for a frame $L$, the operator $cbd_L:L\to L$, defined by \[\displaystyle{cbd_L(a)=\bigwedge\big\{x\in L\mid a\leq x \text{ and } (x\to a)=a\big\},}\] is called the Cantor-Bendixson derivative on $L$, and that $[a,cbd_L(a)]$ is the largest Boolean interval above $a$ (see, e.g. \cite{MR3363833}). 

\begin{prop}\label{c33}
The frame $\mathcal{L}(\mathbb{Z}_{p})$ satisfies \[cbd_{\mathcal{L}(\mathbb{Z}_{p})}(0)=0\] 

\end{prop}
\begin{proof}
Suppose $cbd_{\mathcal{L}(\mathbb{Z}_{p})}(0)\neq 0$. Then, there is $B_r(a)\in \mathcal{L}(\mathbb{Z}_{p})$ be such that $B_r(a)\leq cbd(0)$. It follows that $[0,B_r(a)]$ is boolean since $[0,cbd_{\mathcal{L}(\mathbb{Z}_{p})}(0)]$ is boolean, but this contradicts the fact that $[0,B_r(a)]\cong \mathcal{L}(\mathbb{Z}_{p})$. 
\end{proof}

Now we will see the metric uniformity of $\mathcal{L}(\mathbb{Z}_p)$.
For each natural number $n$, set 
\begin{center}
$U_n=\{B_r(a)\in \mathcal{L}(\mathbb{Z}_p)\mid a\in \mathbb{Z}, \, r=p^{-n-i}, \, i=0,1,2,... \}$.
\end{center}
Then each $U_n$ is a cover of $\mathcal{L}(\mathbb{Z}_p)$. Note that by (R$_5$), any $B_s(b)$ can be written as a finite join of elements of $U_n$.
\begin{prop} $\{U_n \mid n\in \mathbb{N} \}$ is a basis for a uniformity $\,\mathcal{U}$ on $\mathcal{L}(\mathbb{Z}_p)$.
\end{prop}
\begin{proof}
First, note that $U_{n+1}$ is contained in $U_n$ and thus $U_{n+1}\leq U_n$, that is, $\{U_n \, | \, n\in \mathbb{N} \}$ is a filter basis of covers.\\
 Now, for any $B_r(a)\in U_{n+1}$,
\begin{center}
$U_{n+1}B_r(a)=\bigvee \{B_s(b)\mid s=p^{-n-i} \text{ with } i\in \mathbb{N}, \, B_s(b)\wedge B_r(a)\neq 0\}\leq B_r(a)$,
\end{center}
thus $U_{n+1}^*\leq U_n$ (or $U_{n+1}U_{n+1}\leq U_n$).\\
Finally, if $B_s(b)\leq B_r(a)$, with $r=p^{-n}$. Then
\begin{center}
$U_{n+1}B_s(b)=\bigvee \{B_t(c)\mid t=p^{-n-i} \text{ with } i\in \mathbb{N}, \, B_t(c)\wedge B_s(b)\neq 0\}\leq B_s(b)$,
\end{center}
therefore $B_s(b) \lhd_{\mathcal{U}} B_r(a)$. Then the admissibility condition follows from (Q$_3$) and its remark.
\end{proof}
\begin{cor} $\mathcal{L}(\mathbb{Z}_p)$ is metrizable.
\end{cor}

As in the topological case the space  $\mathbb{Q}_p$ is the completion of $\mathbb{Q}$ with respect to the $p$-adic absolute value $|\cdot|_p$, the $p$-adic integers $\mathbb{Z}_p$ is a subring of $\mathbb{Q}_p$. Additionally, the ring $\mathbb{Z}_p$ is complete and $\mathbb{Z}$ is dense in $\mathbb{Z}_p$; that is, if $x\in \mathbb{Z}_p$, then there exists a sequence $\{x_n\}$ of integers converging to $x$.

These facts have its point-free context:

Recall that a uniform frame $(L,\mathcal{U})$ is complete if each ue-surjection (uniform embedding surjection) $h:(M,\mathcal{V}) \to (L,\mathcal{U})$ is an isomorphism.

\begin{prop} \label{completion}
The uniform frame $(\mathcal{L}(\mathbb{Z}_p),\mathcal{U})$ is complete.
\end{prop}
\begin{proof}
Let $h\colon(M,\mathcal{V}) \to (\mathcal{L}(\mathbb{Z}_p),\mathcal{U})$ be any dense ue-surjection. Since $h$ is dense, we just need to prove that it has a right inverse. Let $h_*$ be the right (Galois) adjoint of $h$. Note that $hh_*=id$ since $h$ is onto. Therefore, it is enough to show that $h_*$ is a frame homomorphism, that is, $h_*$ (defined on generators) turns the conditions (R$_1$)$-$(Q$_3$) into identities on $M$.\\
(R$_1$) and (Q$_1$): These are obvious since $h_*$ preserves arbitrary meets.\\
(Q$_2$): Since $h_*[U_n]$ is a cover of $M$ for each $n$, then 
\begin{center}
$1=\bigvee\big\{h_*\big(B_r(a)\big)\mid a\in \mathbb{Z}, r\in|\mathbb{Z}|\big\}$.
\end{center}
(Q$_3$): Given $b\in\mathbb{Z}$ with $|a-b|_p<r$ and $s=p^{-m}<r$, then $B_{r}(a)=B_{r}(b)$. Take any $B_t(c)\in U_{m+1}$, then
\begin{equation*}
B_{r}(a)\wedge B_t(c)=B_{r}(b)\wedge B_t(c)\leq B_s(b).
\end{equation*}
Thus,  \[h_*\big(B_{r}(a)\big)\wedge h_*\big(B_t(c)\big)\leq h_*\big(B_s(b)\big)\leq\bigvee\big\{h_*\big(B_s(b)\big)\mid |a-b|_p<r, \,s<r\big\}.\] \\Since $h$ is a dense ue-surjection, $h_*[U_{m+1}]\in \mathcal{V}$ is a cover. Therefore, 
\begin{center}
$h_*\big(B_{r}(a) \big)\leq\bigvee\big\{h_*\big(B_s(b)\big)\mid |a-b|_p<r, \,s<r\big\}$.
\end{center}
Since the reverse inequality clearly holds, this proves that
\begin{center}
$h_*\big(B_{r}(a) \big)=\bigvee\big\{h_*\big(B_s(b)\big)\mid |a-b|_p<r, \,s<r\big\}$. 
\end{center}

\end{proof}
Note that the map $B_r(a)\mapsto B_r\langle a \rangle$ induces a canonical frame homomorphism $h\colon\mathcal{L}(\mathbb{Z}_p)\to \Omega(\mathbb{Z})$. Since $h$ maps each cover $U_n$ to the analogous cover \[\displaystyle{V_n=\{B_r\langle a \rangle \in \Omega(\mathbb{Z})\mid a\in \mathbb{Z}, \, r=p^{-n-i}, \, i=0,1,2,... \}}\] of $\Omega(\mathbb{Z})$, it follows that the uniformity on $\Omega(\mathbb{Q})$ induced by the metric uniformity $\, \mathcal{U}$ on $\mathcal{L}(\mathbb{Z}_p)$ via $h$ is the metric uniformity $\mathcal{Z}$ on the integers with respect to the $p$-adic absolute value, and $h\colon(\mathcal{L}(\mathbb{Z}_p),\mathcal{U})\to(\Omega(\mathbb{Z}),\mathcal{Z})$ is a uniform homomorphism.
\begin{cor}
$h\colon(\mathcal{L}(\mathbb{Z}_p),\mathcal{U})\to(\Omega(\mathbb{Z}),\mathcal{Z})$ is a completion of $(\Omega(\mathbb{Z}),\mathcal{Z})$.
\end{cor}

\begin{cor}\label{c3}
The frame $\mathcal{L}(\mathbb{Z}_{p})$ is completely uniform

\end{cor}



It can be shown that $\mathcal{L}(\Z_p)$ is a spatial frame whose space of points is homeomorphic to $\Z_p$ (see, e.g. \cite{avila2020frame}). Also, any uniform frame is regular, and a compact regular frame is spatial, thus we can conclude the following:

\begin{thm}\label{c2}

Let $L$ be a frame, then \[L\cong\mathcal{L}(\mathbb{Z}_{p})\Leftrightarrow L\text{ is 0-dimensional, compact, and metrizable, with } cdb_L(0)=0.\]
\end{thm}

\section{A point-free version of a theorem of Hausdorff-Alexandroff}\label{haus}

In this section we will produce a point-free counterpart of the following well known theorem due to Hausdorff and Alexandroff:\[\emph{ Every compact metric space is a continuous image of Cantor space }\]

As noted in \citep{Alain} the map $\phi:\mathbb{Z}_2 \to [0,1]$ where
$$\phi\left( \sum_{i \ge 0}b_i2^i\right) = \sum_{i \ge 0} \frac{b_i}{2^{i+1}}.$$
is continuous and onto. We introduce some topological and set theoretic properties of this map that motivates a frame-theoretic analogy.

\begin{prop}
Let $u$ be an integer coprime to $2$, $g$ a nonnegative integer where $0 < u < 2^g$, then $\displaystyle{\;\phi^{-1}\left(\frac{u}{2^g}\right)\;}$ has exactly two elements. 
\end{prop}
\begin{proof}
By the division algorithm there are unique integers $b_0,\ldots,b_{g-1} \in \{0,1\}$ such that 
$$u = \sum_{0 \le i < g} b_i2^i,$$ we note that $b_0 = 1$. Define for each $0\le i < g$, $a_i = b_{g-(i+1)}$, then 
$$\frac{u}{2^g} = \frac{\sum_{0 \le i < g} b_i2^i}{2^g} = \sum_{0 \le i < g}\frac{b_i}{2^{g-i}}=\sum_{0 \le i < g}\frac{b_{g-(i+1)}}{2^{i+1}}
=\sum_{0 \le i < g}\frac{{a_i}}{2^{i+1}}.$$ Define $$x_+ = \sum_{0 \le i < g}a_i2^i$$ and $$x_- = \sum_{0 \le i < g-1}a_i2^i + \sum_{g<i}2^i$$
then $\phi(x_+) = \frac{u}{2^g}$ and $$\phi(x_-) = \frac{u-1}{2^g} + \sum_{g < i}\frac{1}{2^{i+1}} = \frac{u-1}{2^g} + \frac{1}{2^g}\sum_{0<i}\frac{1}{2^i}=\frac{u-1}{2^g} + \frac{1}{2^g} = \frac{u}{2^g}.$$
Note that $x_-\neq x_+$ and they satisfy the conclusion on lemma 2, so they are the only two dyadic integers that whose image is $u/2^g$.   
\end{proof}

For the reminder of this paper, for every basic element $B$ of $\mathcal{L}(\mathbb{Z}_2)$, we denote the length of $B$ is the unique natural number $n$ with $B=B_{1/2^n}(a)$ for some $2$-adic integer $a$. For simplicity, we write $B(a,n)$ instead of $B_{1/2^n}(a)$.

\begin{lem}
For $\phi\colon\mathbb{Z}_2 \to [0,1]\subset \mathbb{R}$ as above, we have
$$\phi^{-1}(-\infty,u/2^g) = \bigcup\left\{B(a,{g+k}) \mid \phi(a) < \frac{u}{2^g} - \frac{1}{2^{g+k}} \text{ and } k \geq 0 \right\}$$
and 
$$\phi^{-1}(u/2^g,+\infty) = \bigcup\left\{B(a,{g+k}) \mid \phi(a) > \frac{u}{2^g} + \frac{1}{2^{g+k}} \text{ and } k \geq 0 \right\}$$
for every integer $u$ and nonnegative integer $g$.
\end{lem}
\begin{proof}
Assume without loss of generality that $(u,2) = 1$. Now, suppose that $z \in \phi^{-1}(-,u/2^g)$, then $\phi(z) < u/2^g$. Write $z = \sum_{i \ge 0} z_i2^i$ and consider the partial sum of its first $g$ dyadic digits $z' = \sum_{i=0}^{g-1}z_i2^i$, note that $$\phi(z') = \sum_{i=0}^{g-1}\frac{z_i}{2^{i+1}} = \frac{1}{2^g}\sum_{i=0}^{g-1}z_i2^{g-(i+1)}$$
set $\widetilde z = \sum_{i=0}^{g-1}z_i2^{g-(i+1)}$, we now proceed by cases comparing the integers $\widetilde z$ and $u$
\begin{description}
	\item[Case 1: $\widetilde z \ge u$.] We get $u/2^g > \phi(z) \ge \phi(z') \ge \widetilde z/2^g$ which is impossible.
	\item[Case 2: $\widetilde z < u-1$.] Notice that $$\phi(z') = \frac{\widetilde z}{2^g} < \frac{u-1}{2^g}=\frac{u}{2^g} - \frac{1}{2^{g+0}}$$ thus $z \in B(z',g).$
	\item[Case 3: $\widetilde z = u-1$.] See that it is impossible that $z_i = 1$ for all $i \ge g$; otherwise, $$\phi(z) = \frac{\widetilde z}{2^g} + \sum_{i=g}^{\infty}\frac{1}{2^{i+1}} = \frac{u-1}{2^g} + \frac{1}{2^g} = \frac{u}{2^g}.$$
so let $m \ge g$ be smallest such that $z_i = 1$ for all $g \le i < m$, note that both $z_{g-1}$ and $z_{m}$ are $0$. Now set $z'' = z' + \sum_{i=g}^{m-1}2^i$, and set $k = g - (m-1)$, then
\begin{align*}
\phi(z'') &= \frac{\widetilde z}{2^g} + \frac{1}{2^g}\frac{2(2^{k-1}-1)}{2^k}\\
          &= \frac{u-1}{2^g} + \frac{2^k-2}{2^{g+k}}\\
					&= \frac{u}{2^g} -\frac{2^k}{2^{g+k}} + \frac{2^k - 2}{2^{g+k}}\\
					&= \frac{u}{2^g} - \frac{2}{2^{g+k}}\\
					&< \frac{u}{2^g} - \frac{1}{2^{g+k}} 
\end{align*}
thus $z \in B(z'',g+k)$.

Now suppose that $a \in \mathbb{Z}_2$, $k \ge 0$ with $\phi(a) < \frac{u}{2^g} - \frac{1}{2^{g+k}}$, and further suppose that $z \in B(a,{g+k})$; write $z = \sum_{i=0}^{\infty}z_i2^i$ and set $z'' = \sum_{i=0}^{g+k-1}z^i2^i$, note that $z''$ and $a$ have the same dyadic expansion up to the $(g+k-1)$-th digit, then
\begin{align*}
\phi(z) &\le \phi(z'' + \sum_{i=g+k}^{\infty}2^i)\\
        &= \phi(z'') + \sum_{i=g+k}^{\infty}\frac{1}{2^{i+1}}\\
				&\le \phi(a) + \sum_{i=g+k}^{\infty}\frac{1}{2^{i+1}}\\
				&< \frac{u}{2^g} - \frac{1}{2^{g+k}} + \sum_{i=g+k}^{\infty}\frac{1}{2^{i+1}}\\
				&=\frac{u}{2^g} - \frac{1}{2^{g+k}} + \frac{1}{2^{g+k}}\\
				&=\frac{u}{2^g}.
\end{align*}
summarizing $\phi(z) < \frac{u}{2^g}$, then $z \in \phi^{-1}(-\infty, u/2^g)$.
\end{description}
The second equality follows similarly by a dual argument, avoiding an infinite expansion of $0$'s, as we avoided an infinite expansion of $1's$.
\end{proof}

By analogy we define a frame map  $\widetilde \phi\colon \mathcal{L}[0,1]\rightarrow\mathcal{L}(\mathbb{Z}_2)$ on basic members of $\mathcal{L}[0,1]$ where
\[\widetilde\phi(-,u/2^g) = \bigvee\left\{B(a,{g+k})\mid\phi(a) < \frac{u}{2^g} - \frac{1}{2^{g+k}} \text{ and } k \geq 0 \right\}\]
\[\widetilde\phi(u/2^g,-) = \bigvee\left\{B(a,{g+k})\mid\phi(a) > \frac{u}{2^g} + \frac{1}{2^{g+k}} \text{ and } k \geq 0 \right\}\]
we now proceed to show that $\widetilde \phi$ is an injective frame morphism.

\begin{lem}\label{mor1}
Let $q$ be a dyadic rational, then $\widetilde \phi(-,q) = 1$ if and only if $q>1$.
\end{lem}
\begin{proof}
Suppose that $\widetilde \phi(-,q) = 1$, now write $q = u/2^g$, then
$$\bigvee\left\{B(a,{g+k}) \lvert \phi(a) < \frac{u}{2^g} - \frac{1}{2^{g+k}} \text{ and } k \geq 0 \right\} = 1$$
it follows from compactness of $\mathcal{L}(\mathbb{Z}_2)$ that there are finitely many open balls such that 
$$\bigvee_{i=1}^n B(a_i,{g+k_i}) = 1.$$ Now for the sake of contradiction assume that for all $i = 1,\ldots,n$ we have $$1-\frac{1}{2^{g+k_i}} > \phi(a_i)$$ from the equations
$$1 = \sum_{j=0}^{g+k_i-1}\frac{1}{2^{j+1}} + \sum_{j = g+k_i}^{\infty}\frac{1}{2^{j+1}} = \sum_{j=0}^{g+k_i-1}\frac{1}{2^{j+1}} + \frac{1}{2^{g+k_i}}$$
it follows that for each $i$ there is a  $0 \le j \le g+k_i-1$ with $(a_i)_j = 0$; otherwise, for some $i'$, $(a_{i'})_j = 1$ for all $0 \le j \le g+k_{i'}-1$ so $$\phi(a_{i'}) \ge \phi\Bigg(\sum_{j=0}^{g+k_{i'}-1}2^i\Bigg) = 1-\frac{1}{2^{g+k_{i'}}}$$ which is impossible. Now, let $K$ be the maximum among the $g+k_i$ to define $o = \sum_{j=0}^{K}2^j$, then $|o-a_i|_2 > \frac{1}{2^{g+k_i}}$ which implies that $B(o,K) \wedge B(a_i,{g+k_i}) = 0$, note that $B(o,K) \ne 0$
\begin{align*}
B(o,K) &= B(o,K) \wedge 1\\ 
           &= B(o,K) \wedge \bigvee_{i=1}^n B(a_i,{g+k_i})\\
					 &= \bigvee_{i=1}^n B(o,K) \wedge B(a_i,{g+k_i})\\
					 &= \bigvee_{i=1}^n 0\\
					 &= 0
\end{align*}
contradiction. It follows that for some $i$ 
$$1 - \frac{1}{2^{g+k_i}} \ge \phi(a_i) < \frac{u}{2^g} - \frac{1}{2^{g+k_i}}$$
therefore
$$1 < \frac{u}{2^g}.$$
For the converse, assume that $q > 1$ and note that $\phi(0) = 0 < q - 1/2$ and $\phi(1) = 1/2 < q - 1/2$ then
$$\widetilde \phi(-,q) = \bigvee\left\{B(a,{g+k}) \lvert \phi(a) < q - \frac{1}{2^{g+k}} \text{ and } k \geq 0 \right\}$$ 
$$\ge B(0,1) \vee B(1,1) =1.$$
\end{proof}

Note that $\widetilde \phi(p,q) = \widetilde \phi((-,q) \wedge (p,-)) = \widetilde\phi((-,q)) \wedge \widetilde\phi((p,-))$ so we get the following
\begin{cor}\label{mor2}
For $p$ and $q$ dyadic rationals, $\widetilde \phi(p,q) = 1$ if and only if $p<0$ and $q>1$.
\end{cor}


\begin{lem}\label{mor3}
For $p$, $q$, $r$, and $s$ dyadic rationals, $\widetilde \phi\left( (p,q) \vee (r,s) \right) = 1$ if and only if $(p,q) \vee (r,s) = 1$.
\end{lem}
\begin{proof}
Assume that $(p,q) \vee (r,s) = 1$, since $\widetilde \phi$ is a morphism, we get that $$\widetilde\phi((p,q) \vee (r,s)) = \widetilde \phi(1) = 1.$$

On the other hand, set $m = \min(p,r)$ and $M = \max(q,s)$ the either $(p,q) \land (r,s) \neq 0$ or $(p,q) \land (r,s) \neq 0$: in the first case $(p,q) \vee (r,s) = (m,M)$ and thus the result follows from the above corollary, in the later, either $q < r$ or $s < p$, set $a$ to be the average of $q$ and $r$ or $s$ and $p$ depending on the case. We see that $(p,q) \vee (r,s) \le (-,a) \vee (a,-)$, then $1=\widetilde\phi((p,q) \vee (r,s)) \le \widetilde\phi((-,a) \vee (a,-))$ so $\widetilde\phi((-,a)) \vee \widetilde\phi((a,-))=1$, write $a = u/2^g$ and
thus 
\begin{align*}
&1=\\
&\bigvee\left\{B(x,{g+k}) \,|\, \phi(x) < a - 1/2^{g+k} \right\} \vee \bigvee\left\{B(x,{g+k}) \,|\, \phi(x) > a + 1/2^{g+k} \right\} 
\end{align*} 
\end{proof}

It follows from \ref{reals} that $\widetilde{\phi}$ is a frame morphism.

\begin{lem}\label{lem1}
If $\widetilde{\phi}(u/2^{g},v/2^{h})=0$ then $(u/2^{g},v/2^{h})=0$.
\end{lem}

\begin{proof}
By contrapositive. Assume that $(u/2^{g},v/2^{h})\neq 0$ then $u/2^{g}<v/2^{h}$ now set $m=\text{max}\{g,h\}$ and define $u', v'$ such hat $\dfrac{u}{2^{g}}=\dfrac{u'}{2^{m}}$ and $\dfrac{v}{2^{h}}=\dfrac{v'}{2^{m}}$. Now, set $w=\dfrac{u'+v'}{2^{m+1}}$, that is, $w$ is the average, which is a dyadic rational. Since $\dfrac{u'}{2^{m}}<w<\dfrac{v'}{2^{m}}$ there is a unique $2$-adic integer $\bar{w}$ which is eventually zero in its $2$-adic expansion so that $\phi(\bar{w})=w$.
Now, set $B=B(\bar{w}, 1/2^{m+2})$ and note that

\begin{align*}
\frac{v'}{2^{m}}-\frac{1}{2^{m+2}}-\frac{u'+v'}{2^{m+1}}&=\frac{4v'-1-2u'-2v'}{2^{m+1}} \\
&=\frac{2v'-2u'-1}{2^{m+1}} \\
&=\frac{2(v'+u')-1}{2^{m+1}}\geq\frac{2(1)-1}{2^{m+1}}>0\\
\end{align*}

Then $\phi(\bar{w})=w<\dfrac{v'}{2^{m}}-\dfrac{1}{2^{m+2}}$. Similarly, we have that 

\begin{align*}
w-\frac{u'}{2^{m}}-\frac{1}{2^{m+2}}&=\frac{2(v'-u')-1}{2^{m+1}}\\
 &\geq \frac{1}{2^{m+1}}>0\\
\end{align*}

thus $\phi(\bar{w})=w>\frac{u'}{2^{m}}+\frac{1}{2^{m+2}}$. Then $B$ is a member of 
$$\Big\{B(a,h+k)\mid \phi(a)<\frac{v}{2^{h}}-\frac{1}{2^{g+k}}\text{ with } k\geq 0\Big\}$$ 
and of 
\[\Big\{B(a,h+k)\mid \phi(a)>\frac{v}{2^{h}}+\frac{1}{2^{g+k}}\text{ with } k\geq 0\Big\}\]
 clearly $B\leq\widetilde{\phi}(-, v/2^{h})$ and $B\leq\widetilde{\phi}(u/2^{g},-)$ and $B\neq\emptyset$, thus \[0<B\leq\widetilde{\phi}(-, v/2^{h})\wedge\widetilde{\phi}(u/2^{g},-)=\widetilde{\phi}(u/2^{g},v/2^{h}).\]

\end{proof}

\begin{cor}\label{mor4}
The morphism $\widetilde{\phi}\colon\mathcal{L}[0,1]\rightarrow\mathcal{L}(\mathbb{Z}_{2})$ is injective.

\end{cor}

\begin{proof}
Since $\mathcal{L}[0,1]$ and $\mathcal{L}(\mathbb{Z}_{2})$ are regular and $\mathcal{L}(\mathbb{Z}_{2})$ is compact (see \cite{avila2020frame}) it suffices to show that $\widetilde{\phi}$ is dense. 
Let $a\in\mathcal{L}[0,1]$ and assume that $\widetilde{\phi}(a)=0$. Then 
\[0=\widetilde{\phi}(a)=\tilde{\phi}\Bigg(\bigvee_{i}(p_{i},q_{i})\Bigg)=\bigvee_i\tilde{\phi}(p_{i},q_{i})\] so $\widetilde{\phi}(p_{i},q_{i})=0$ for each $i$, thus by the above Lemma \ref{lem1}, $(p_{i},q_{i})=0$ therefore $a=0$.
\end{proof}

Now, consider a morphism $\varphi\colon L\rightarrow M$ between frames. For any set $\EuScript{K}$, take the coproduct $\bigoplus_{\EuScript{K}}L={^{\EuScript{K}}}L$ and $\bigoplus_{\EuScript{K}}M={^{\EuScript{K}}}M$. For each $j \in \EuScript{K}$, we have a morphisms $\iota_{j}\colon L\rightarrow\mathcal{D}(M^{\EuScript{K'}})$  given by $\mu\alpha\kappa_{j}$, where $\kappa_{j}\colon M\rightarrow  M^{\EuScript{K'}}$ (the morphisms into the coproduct as semilattices), $\alpha\colon M^{\EuScript{K'}}\rightarrow\mathcal{D}(M^{\EuScript{K'}})$ ($\alpha(x)=\downarrow x$), and $\mu\colon\mathcal{D}(M^{\EuScript{K'}})\rightarrow {^{\EuScript{K}'}}M$ is the quotient given by the relation. Thus, by the universal property of the coproduct, we have a unique frame morphism $\varphi^{\flat}\colon ^{\EuScript{K}}L\rightarrow {^{\EuScript{K}}}M$. The morphims $\varphi^{\flat}$ is constructed as follows:\\ First, we have a morphism $\hat{\varphi}\colon L^{\EuScript{K}'}\rightarrow\mathcal{D}(M^{\EuScript{K}'})$ given by $\hat{\varphi}(x_{i})=\bigwedge_{k=1}^{n}\varphi_{j_{k}}(x_{j_{k}})$, where $\varphi_{j_{k}}=\alpha\kappa_{j}$. Then, by the universal property of the completion $\mathcal{D}(L^{\EuScript{K}'})$, we have a frame morphism $\varphi^{\sharp}\colon \mathcal{D}(L^{\EuScript{K}'})\rightarrow \mathcal{D}(M^{\EuScript{K}'})$. Thus, for each $j$, we have a frame morphism $\mu\varphi^{\sharp}\alpha'_{j}\colon L\rightarrow {^{\EuScript{K}}}M$. That is, $$\varphi^{\flat}(U)=\bigvee\{\mu\varphi^{\sharp}(U)\}=\bigvee\{\mu\hat{\varphi}(x_{i})\mid x_{i}\in U\}.$$ With this preamble at hand we have the following.

\begin{lem}\label{inj}
Let $\varphi\colon L\rightarrow M$ a dense morphism between compact regular frames then \[\varphi^{\flat}\colon {^{\EuScript{K}}}L\rightarrow {^{\EuScript{K}}}M\] is an injective frame morphism.

\end{lem}

\begin{proof}

It is enough to see that $\varphi^{\flat}$ is a dense morphism. Let $U\in L^{\EuScript{K}'}$ such that $\varphi^{\flat}(U)=0$ that is, $\bigvee\{\mu\hat{\varphi}(x_{i})\mid x_{i}\in U\}=0$ thus every $\mu\hat{\varphi}(x_{i})=0$ since $\mu$ is the corresponding frame morphism given by a nucleus (see \cite[III.11]{PicadoPultr}), we have $\hat{\varphi}(x_{i})=0$. Thus, there is an $i$ such that $\varphi(x_{i})=0$. Since $\varphi$ is dense, we have $x_{i}=0$ and thus $U=0$ as required.

\end{proof}

\begin{cor}\label{de}

There is an injective frame morphism \[\bar{\phi}\colon {^{\EuScript{K}}}\mathcal{L}[0,1]\rightarrow{^{\EuScript{K}}}\mathcal{L}(\mathbb{Z}_{2})\]

\end{cor}

\begin{proof}

By corollary \ref{mor4} we have an dense morphism $\widetilde{\phi}\colon\mathcal{L}[0,1]\rightarrow\mathcal{L}(\mathbb{Z}_{2})$ then by the Lemma \ref{inj} we have the conclusion.

\end{proof}

\begin{prop}\label{cantor}
Any countable coproduct of $\mathcal{L}(\mathbb{Z}_{2})$ is isomorphic to $\mathcal{L}(\mathbb{Z}_{2})$.

\end{prop}

Let us call $ {^{\mathbb{N}}}\mathcal{L}[0,1]=\EuScript{H}$ the \emph{ Hilbert cube}. Thus by corollary \ref{de} and proposition \ref{cantor} we have:

\begin{prop}\label{end}

There is an injective frame morphism $\eta\colon\EuScript{H}\rightarrow\mathcal{L}(\mathbb{Z}_{2})$

\end{prop}

As in the topological case we have the following.

\begin{prop}\label{quohil}
The following are equivalent for a frame $L$:
\begin{itemize}
\item[a)] $L$ is regular and has a countable basis.
\item[b)] $L$ is a quotient of a $\EuScript{H}$.
\end{itemize}
\end{prop}

\begin{proof}
a) $\rightarrow$ b): Let $B$ be a countable basis for $L$, and let $$S = \{(a,b) \in B^2\, | \, a^* \vee b = 1\}.$$ Clearly $S$ is countable and, since $L$ is regular and Lindel\"of then it is normal, so for every pair $(a,b) \in S$ there is a frame morphism $h_{a,b}\colon\mathcal{L}[0,1] \to L$ where $h_{a,b}((0,-)) \leq a^*$ and $h_{a,b}((-,1)) \le b$, see \cite{Dowker}.\\ Moreover, $a \land h_{a,b} ((0,-)) \le a \land a^*  = 0$ so $a \land h_{a,b}((0,-)) = 0$, then
\begin{align*}
a &= a \land 1\\
  &= a \land [h_{a,b}(1)]\\
	&= a \land [h_{a,b}((0,-) \lor (-,1))]\\
	&= a \land [h_{a,b}((0,-)) \lor h_{a,b}((-,1))]\\
	&= [a \land h_{a,b}((0,-))] \lor [a \land h_{a,b}((-,1))]\\
	&= 0 \lor [a \land h_{a,b}((-,1))]\\
	&= a \land h_{a,b}((-,1))
\end{align*}
thus $a \le h_{a,b}((-,1)) \le b$.

Set $H = \{h_{a,b}\,\mid\, (a,b) \in S\}$ and for each $h$ in $H$, $\mathcal{L}[0,1]_h$ is a copy of $\mathcal{L}[0,1]$, by universal properties of the coproduct $\bigoplus_{h \in H}\mathcal{L}[0,1]_{h}$ there is a unique morphism $p\colon \bigoplus_{h \in H}\mathcal{L}[0,1]_{h} \to L$ such that $h_{a,b} = p \circ i_{h_{a,b}}$ where the $i_{h_{a,b}}$ are the inclusion morphisms of the coprodct. Now, let us show that $p$ is indeed surjective, so it suffices to show that every element $b$ in the basis $B$ is an image of a member of the coproduct. Set $s = \bigvee\{i_{h_{a,b}}(-,1)\,|\, (a,b) = S\}$, then $p(s) = \bigvee\{h_{a,b}(-,1)\,|\, (a,b) = S\}$. By construction of $h_{a,b}$, we have that $h_{a,b}(-,1) \le b$ for every $a$ in $B$ so $b \ge \bigvee\{h_{a,b}(-,1)\,|\, (a,b) = S\}$. On the other hand, by regularity of $L$
\begin{align*}
b &=   \bigvee\{x \in B\,|\, x \prec b\}\\
  &\le \bigvee\{h_{x,b}(-,1)\,|\, x \prec b\}\\
	&\le \bigvee\{h_{a,b}(-,1)\,|\, (a,b) \in S\}
\end{align*}
thus $p(s) = b$.

b) $\rightarrow$ a): We know that $\mathcal{L}[0,1]$ is both regular and has a countable basis, the countable coproduct of copies of $\mathcal{L}[0,1]$ is both regular (see \cite[VII. 1.2 and 4.5]{PicadoPultr} ) and possesses a countable basis. Moreover, these two properties pass to quotients (sublocales of regular locales are regular).

\end{proof}

In particular if $L$ is compact and since $\EuScript{H}$ is a compact regular frame (page 137 of \cite{PicadoPultr}) then $L$ is a closed quotient (see \cite[VII. 2.2.3]{PicadoPultr}) .

We are going to prove the point-free version of the Hausdorff-Alexandroff theorem, first we need the following material.
\begin{dfn}

Let $L$ and $M$ frames, we say that $L$ is \emph{retract} of $M$ if, there is a surjective frame morphism $\rho\colon M\rightarrow L$ such that there exists a morphism $\iota\colon L\rightarrow M$ such that $\rho\iota=id_{L}$.

\end{dfn}
If $L$ is a retract of $M$ then the composition $\iota\rho$ is an idempotent frame morphism.

\begin{lem}\label{down}
Let $b=B(a,n)$ and $x$ be in $\mathcal{L}(\mathbb{Z}_2)$ where $b$ is a basic element and $x < 1$. If $b \not\le x$, then there is a basic element $b'=B(a',n+1) < b$ such that $b' \not\le x$.  
\end{lem}
\begin{proof}
Let $b$ and $x$ be as above, and assume for the sake of contradiction that for every $B(a',n+1) < b$, $B(a',n+1) \le x$, then $\bigvee_{B(a',n+1) < b}B(a',n+1) \le x$, but $\bigvee_{B(a',n+1) < b}B(a',n+1)=b \not\le x$ contradiction.
\end{proof}

\begin{lem}\label{fbasic}
Let $\mathcal{B}$ the set of basic elements of $\mathcal{L}(\mathbb{Z}_2)$ and $x < 1$, then there is a function $f: \mathcal{B} \to \mathcal{B}$ such that $f(b) \not\le x$ for all $b \in \mathcal{B}$. Moreover, $f$ can be chosen so that $f$ preserves
the radius of each ball in $\mathcal{B}$. 
\end{lem}
\begin{proof}
First, we define $f$ recursively on the exponent of the radius of each basic element $b$. Set $f(B(a,0)) = B(a,0)$. Now suppose that $f(B(a,n))$ is defined, if $B(a,n+1) \not\le x$ set $f(B(a,n+1)) = B(a,n+1)$; otherwise, by \cref{down}, there is a basic element $B(a',n+1) < f(B(a,n))$ such that $B(a',n+1) \not\le x$, so set $f(B(a,n+1)) = B(a',n+1)$.

Second, we proceed by induction on the exponent $n$ of the radius of basic elements. Let $B(a,n) \in \mathcal{B}$. When $n=0$, $B(a,n)=B(a,0)$ so $f(B(a,n)) = B(a,0)=B(a,n)$. Now let $k\ge 0$ and suppose that $f(B(a,k)) = B(a',k)$, then either $f(B(a,k+1)) = B(a,k+1)$ or $f(B(a,k+1)) = B(a'',k+1) < f(B(a,k)) = B(a',k)$.  
\end{proof}

\begin{prop}\label{endo}
Let $x$ be in $\mathcal{L}(\mathbb{Z}_2)$ where $x <1$, and let $\mathcal{B}$ the set of basic elements of $\mathcal{L}(\mathbb{Z}_2)$. The function $h\colon\mathcal{B} \to \mathcal{L}(\mathbb{Z}_2)$ given by \[h(b) = \bigvee f^{-1}(\{b\})\] where $f$ is as in \cref{fbasic} extends to a frame endomorphism $H$ on $\mathcal{L}(\mathbb{Z}_2)$ with $x = \bigvee H^{-1}(\{0\})$.
\end{prop}

\begin{proof}
First, we need to show that $h(1) = \bigvee \{h(B(a,n)) \mid B(a,n) < 1\}$. Since $f(1) = 1$ we have $h(1) \ge 1 \ge h(1)$. For each positive integer $i$ set \[R_i := \{B \in \mathcal{B}\mid B=B(a,i)\},\] then $R_i=\{B(k,i)\mid k=0\ldots 2^i-1\}$. Note, that $1=\bigvee R_1$, so either
$f(B(0,1))=B(0,1)$ and $f(B(1,1))=B(1,1)$, or $f(B(0,1))=f(B(1,1))= B(0,1)$, or
$f(B(0,1))=f(B(1,1))=B(1,1)$, we cannot have that $f^{-1}(R_1)=\varnothing$; otherwise, $B(0,1),B(1,1) \le x$, so $x \ge B(0,1) \lor B(1,1) =1$ impossible.

Second, let $b=B(a,n)$ be a basic element, then either $b \le x$ or $b \not\le x$
\begin{description}
\item[Case $b \le x$.] Then $f^{-1}(\{b\})=\varnothing$ and for every $b'<b$ we have $b' < x$, so $f^{-1}(\{b'\})=\varnothing$. Thus $h(b)=0$ and $h(b')=0$ for $b'<b$.
\item[Case $b \not\le x$.] We know that for some $b'=B(a',n+1)<b$, $b' \not\le x$; otherwise, $b=\bigvee_{b'<b}b' \le x$ impossible. We know that there are exactly two
basic elements of with radius $2^{-(n+1)}$ below $b$, call them $b_0$ and $b_1$, then
there are two scenarios:
\begin{description}
\item[Subcase] $b_i \not \le x$ for $i=0,1$ So $f^{-1}(\{b_i\}) = \{b_i\}$ for $i=0,1$ so $h(b_i)=b_i$.
\item[Subcase] $b_i \not\le x$, $b_j \le x$ for $i,j=0,1$ and $i\neq j$. Since $f$ preserves radius we have that $f^{-1}(\{b_i\}) = \{b_0,b_1\}$ and $f^{-1}(\{b_j\})=\varnothing$
so $h(b_i) = b_0 \lor b_1$ and $h(b_j)=0$.
\end{description}
In all cases we see that $h(b)=\bigvee\{h(b')\mid b'=B(a',n+1)<b\}$. 
\end{description}

Third, let $B(a,n)$, and $B(a',n')$ be basic elements, and assume $$B(a,n) \land B(a',n')=0.$$ Let us say that $n \le n'$, then we can express 
\[B(a,n) = \bigvee\{B(\bar{a},n') \mid B(\bar{a},n') \le B(a,n) \},\] then
\[B(a,n) \land B(a',n') =  \bigvee\{B(\bar{a},n') \land B(a',n') \mid B(\bar{a},n') \le B(a,n) \} = 0.\] So it suffices to show that $h(b) \land h(b') = 0$ whenever $b$, and $b'$ are basic elements of the same length with $b \land b'=0$. Thus, let $B(a,n)$ and $B(a',n)$ be basic elements where $B(a,n) \land B(a',n)=0$, then $$f^{-1}(\{B(a,n)\}) \cap f^{-1}(\{B(a',n)\}) = \varnothing;$$ otherwise, $f$ is not a function. Since $f$ preserves length, all the members of $f^{-1}(\{B(a,n)\})$ and of $f^{-1}(\{B(a',n)\})$ have length $n$. If one of these sets is empty there is nothing to prove, so suppose that there are $b \in f^{-1}(\{B(a,n)\})$ and $b' \in f^{-1}(\{B(a',n)\})$, since $b \neq b'$ and the have the same length we have that $b \land b' = 0$, thus 
\begin{align*}
h(B(a,n)) \land h(B(a',&n)) = \bigvee f^{-1}\left(\{B(a,n)\}\right) \land \bigvee f^{-1}\left(\{B(a,n)\}\right)\\ 
                          &= \bigvee\{b \land b' \mid b \in f^{-1}(\{B(a,n)\}), b' \in f^{-1}(\{B(a',n)\}) \}\\ 
                          &= \bigvee\{0\}\\
                          &=0
\end{align*}
Therefore, $h$ extends uniquely to a frame endomorphism $H$.

Finally, note that for every basic element $b$ of $\mathcal{L}(\mathbb{Z}_2)$, $f^{-1}(\{b\})=\varnothing$ if and only if $b \le x$, in other words $H(b) = 0$ exactly when $b \le x$. Since $H$ is a frame morphism, we see that $\bigvee H^{-1}(0) = x$.
\end{proof}

\begin{cor}\label{retrac}
Every non-trivial closed quotient of $\mathcal{L}(\mathbb{Z}_2)$ is a retraction of $\mathcal{L}(\mathbb{Z}_2)$.
\end{cor}

\begin{proof}
Let $C$ be a non-trivial quotient of $\mathcal{L}(\mathbb{Z}_2)$, then there is an $x <1$ in $\mathcal{L}(\mathbb{Z}_2)$ with $C\cong \uparrow x$, by \cref{endo} there is an endomorphism $H\colon\mathcal{L}(\mathbb{Z}_2) \to \mathcal{L}(\mathbb{Z}_2)$. Note that $H$ admits a dense-closed quotient factorization $H=r\circ i$ with \[i\colon\mathcal{L}(\mathbb{Z}_{2})\rightarrow\uparrow x\;\;\;r\colon\uparrow x\to\mathcal{L}(\mathbb{Z}_2)\] where $i(a)= a \lor x$ and $r(a)=H(a)$, (see e.g. \cite{banaschewski1999uniform}). Since $\mathcal{L}(\mathbb{Z}_2)$ is compact, (see \cite{avila2020frame}) and closed quotients of compact frames are compact, (see \cite{PicadoPultr}) we note that $i$ and $r$ provide the desired retraction.
\end{proof}

\begin{thm}\label{Alex-Haus}
For every compact metrizable frame $L$ there is an embedding \[\vartheta\colon L\rightarrow \mathcal{L}(\mathbb{Z}_{2}).\] In other words every compact metrizable frame  can be identified as a subframe of the Cantor frame $\mathcal{L}(\mathbb{Z}_{2})$.
\end{thm} 

\begin{proof}
Let $L$ be a compact metrizable frame, then by proposition \ref{quohil} there is a surjective frame morphism $\rho\colon\EuScript{H}\rightarrow L$ also by \ref{end} we have an embedding $\theta\colon\EuScript{H}\rightarrow\mathcal{L}(\mathbb{Z}_{2})$
Then we have the following diagram:
\[\begin{tikzcd}
\EuScript{H} \arrow[r, "\theta"] \arrow[d, "\rho"'] & \mathcal{L}(\mathbb{Z}_{2}) \\ L 
\end{tikzcd}\]

here we identify $L\cong\uparrow x$ for some $x\in\EuScript{H}$. Thus under $\theta$ we have $\theta(\uparrow x)$ is closed in $\mathcal{L}(\mathbb{Z}_{2})$, that is,  $\theta(\uparrow x)=\uparrow\theta(x)$. We have a commutative diagram:

\[\begin{tikzcd}[
  column sep=normal,row sep=small,
  ar symbol/.style = {draw=none,"\textstyle#1" description,sloped},
  isomorphic/.style = {ar symbol={\cong}},
  ]
\EuScript{H} \arrow[r, "\theta"] \arrow[d, "\rho"'] & \mathcal{L}(\mathbb{Z}_{2})\arrow[d,"u"] \\ L \arrow[r, isomorphic] &  \uparrow\theta(x)
\end{tikzcd}\]
where $u$ is the quotient morphism, then by corollary \ref{retrac} we have a morphism $\uparrow\theta(x)\rightarrow\mathcal{L}(\mathbb{Z}_{2})$ which is an embedding.

\end{proof}

\bibliographystyle{plain}

\bibliography{refer}

\end{document}